\documentclass[12pt]{amsart}
\usepackage{amsmath, amssymb, amsfonts, amsthm, mathtools}
\usepackage{hyperref}
\usepackage{tabularx}
\usepackage{booktabs}
\usepackage{caption}
\usepackage{aurical}

\newtheorem{thm}{Theorem}[section]
\newtheorem{lem}[thm]{Lemma}
\newtheorem{cor}[thm]{Corollary}
\newtheorem{prop}[thm]{Proposition}
\newtheorem{ex}[thm]{Example}

\newtheorem*{prob*}{Open problem}

\theoremstyle{definition}

\newtheorem{defi}[thm]{Definition}

\theoremstyle{remark}

\newtheorem{rem}[thm]{Remark}
\newtheorem*{rem*}{Remark}

\DeclareMathOperator{\id}{id}

\newcommand{\kringel}{\mathbin{\raise0.5pt\hbox{$\scriptstyle\circ$}}}
\newcommand{\pkt}{\mathbin{\raise0.5pt\hbox{$\scriptstyle\bullet$}}}
\newcommand{\sq}{\mathbin{\raise0.5pt\hbox{$\scriptscriptstyle\square$}}}

\newcommand{\C}{\mathbb{C}}

\newcommand{\R}{\mathbb{R}}

\newcommand{\ad}{{\rm ad}}

\newcommand{\End}{{\rm End}}

\newcommand{\La}{\mathfrak{a}}
\newcommand{\Lb}{\mathfrak{b}}
\newcommand{\Lc}{\mathfrak{c}}

\newcommand{\Lg}{\mathfrak{g}}

\newcommand{\Ll}{\mathfrak{l}}
\newcommand{\Ln}{\mathfrak{n}}

\newcommand{\Ls}{\mathfrak{s}}

\newcommand{\Lu}{\mathfrak{u}}

\newcommand{\im}{\mathop{\rm im}}

\newcommand{\la}{\lambda}

\newcommand{\ra}{\rightarrow}

\renewcommand{\phi}{\varphi}

\begin{document}


\title[Decomposition of algebras]{Decompositions of algebras and post-associative algebra structures}

\author[D. Burde]{Dietrich Burde}
\author[V. Gubarev]{Vsevolod Gubarev}
\address{Fakult\"at f\"ur Mathematik\\
Universit\"at Wien\\
  Oskar-Morgenstern-Platz 1\\
  1090 Wien \\
  Austria}
\email{dietrich.burde@univie.ac.at}
\address{Fakult\"at f\"ur Mathematik\\
Universit\"at Wien\\
  Oskar-Morgenstern-Platz 1\\
  1090 Wien \\
  Austria}
\email{vsevolod.gubarev@univie.ac.at, wsewolod89@gmail.com}

\date{\today}

\subjclass[2000]{Primary 17B30, 17D25}
\keywords{Post-associative algebra structure, Post-Lie algebra structure, Rota--Baxter operator}

\begin{abstract}
We introduce post-associative algebra structures and study their relationship to post-Lie algebra structures,
Rota--Baxter operators and decompositions of associative algebras and Lie algebras. We show several results
on the existence of such structures. In particular we prove that there exists no post-Lie algebra structure on a pair 
$(\Lg,\Ln)$, where $\Ln$ is a simple Lie algebra and $\Lg$ is a reductive Lie algebra, which is not isomorphic to $\Ln$. 
We also show that there is no post-associative algebra structure on a pair $(A,B)$ arising from a Rota--Baxter
operator of $B$, where $A$ is a semisimple associative algebra and $B$ is not semisimple. The proofs use
results on Rota--Baxter operators and decompositions of algebras. 
\end{abstract}

\maketitle

\section{Introduction}

Post-Lie algebras and post-Lie algebra structures naturally arise in differential geometry, in the study of geometric structures
on Lie groups and crystallographic groups. Existence questions in geometry often can be translated to existence questions
for pre- or post-Lie algebra structures. A well-known example is Milnor's question on the existence of left-invariant affine 
structures on Lie groups. See \cite{BU41,BU51,BU59} for a survey of the questions and results obtained on the
existence and classification of post-Lie algebra structures. \\
On the other hand, these structures also arise in many other areas, such as operad theory, homology of partition sets, universal
enveloping algebras, Yang--Baxter groups, Rota--Baxter operators and $R$-matrices \cite{ELM, VAL}. In particular it is
well-known  \cite{BAG} that Rota--Baxter operators $R$ of weight $1$ on $\Ln$ are in bijective correspondence to
post-Lie algebra structures on pairs $(\Lg,\Ln)$, where $\Ln$ is complete. Recall that a Rota--Baxter operator on an algebra
$A$ is a linear operator $R\colon A\ra A$ satisfying the identity
\[
R(x)R(y)  = R\bigl(R(x)y+xR(y)+\la xy\bigr)  
\]
for all $x,y\in A$ and a scalar $\la$. Such an operator on a Lie algebra always yields a post-Lie algebra structure.
Therefore it is very natural to use Rota--Baxter operators for the existence and classification of
post-Lie algebra structures. We have already obtained several results in \cite{BU59} by using Rota--Baxter operators.
In this paper we obtain further results on post-Lie algebra structures and correct the proof of Proposition
$3.7$ and $3.8$ in \cite{BU59}, which relied on a decomposition theorem of Lie algebras, namely Proposition $3.6$ in \cite{BU59},
which unfortunately is in error. For details see Remark $\ref{4.9}$. 
Decompositions of algebras as a sum of two subalgebras arise naturally from Rota--Baxter operators
on $A$. Here we can use strong theorems on decompositions of Lie and associative algebras by  Onishchik \cite{ON62,ON69}, 
Bahturin and Kegel \cite{BAH}, Koszul \cite{KOS} and others. Instead of post-Lie algebras
one can also consider post-associative algebras \cite{GUK1} and hence also {\em post-associative algebra structures} instead of
post-Lie algebra structures. Again Rota--Baxter operators on an associative algebra yield post-associative algebra structures.
Furthermore, a post-associative structure on a pair of associative algebras induces a post-Lie algebra structure on the
pair of Lie algebras given by commutator. We prove several results on the existence of post-associative structures. \\[0.2cm]
The paper is organized as follows. In section $2$ we introduce the notion of a post-associative algebra structure and recall
the basic definitions concerning post-Lie algebra structures and Rota--Baxter operators. We state the decomposition results for
Lie algebras by  Onishchik, which we will need later on. \\
In section $3$ we prove that every post-associative algebra structure on a pair of associative algebras $(A,B)$, where $B$ 
is semisimple, arises from a Rota--Baxter operator on $B$. We show that, given a post-associative algebra structure on a pair
of semisimple associative algebras $(A,B)$, that the algebras are isomorphic provided one of them is simple. 
Furthermore we show that there are no proper semisimple decompositions of the matrix algebra $M_n(\C)$. \\
In section $4$ we prove the following result in Theorem $\ref{4.1}$. Suppose that there is a post-Lie algebra structure 
on a pair of real or complex Lie algebras $(\Lg,\Ln)$, where $\Ln$ is simple and $\Lg$ is reductive. Then $\Lg$ is also 
simple and both $\Lg$ and $\Ln$ are isomorphic. This generalizes Theorem $3.1$ of \cite{BU59}. Then we give a new proof
of Proposition $3.8$ of \cite{BU59}, which is stated here as Corollary $\ref{4.8}$ and holds for arbitrary fields of
characteristic zero: let $(V,\cdot)$ be a post-Lie algebra structure on a pair $(\Lg,\Ln)$ over a field of 
characteristic zero, where $\Lg$ is semisimple and $\Ln$ is complete. Then $\Ln$ is semisimple. We also give a new proof of 
Proposition $3.7$ of \cite{BU59}
in Corollary $\ref{4.11}$. Finally we obtain results on post-associative algebra structures using a classical theorem on nilpotent
decompositions by Kegel \cite{KEG}.

\section{Preliminaries}

In $2004$, Loday and Ronco \cite{LOR} introduced the notion of a {\em dendriform trialgebra}, also named
{\em post-associative algebra} in \cite{GUK1}, which generalizes the notion of a dendriform algebra.

\begin{defi}
A vector space $V$ over a field $K$ with three  bilinear operations $\succ,\prec, \cdot$ is called a 
{\em post-associative algebra}, if the following identities are satisfied for all $x,y,z\in V$:
\begin{align*} 
(x \prec y) \prec z & = x \prec (y \ast z), \\
(x \succ y) \prec z & = x \succ (y \prec z), \\
(x \ast y) \succ z & = x \succ (y \succ z), \\
x\succ (y\cdot z) & = (x\succ y) \cdot z, \\
(x \prec y) \cdot z & = x \cdot (y \succ z), \\
(x \cdot y) \prec z & = x \cdot (y \prec z), \\
(x \cdot y) \cdot z & = x \cdot (y \cdot z).
\end{align*}
where $x\ast y:=x\prec y+x\succ y+x\cdot y$.
\end{defi}

In $2007$ Vallette \cite{VAL} introduced the notion of a {\em post-Lie algebra} in the context 
of homology of generalized partition posets.

\begin{defi}
A vector space $V$ over $K$ with two bilinear operations $[\, ,]$ and $\cdot$  is called a {\em post-Lie algebra},
if the following identities are satisfied for all $x,y,z\in V$:
\begin{align*} 
[x,y] & = -[y,x], \\
[x,[y,z]] & = - [y,[z,x]]- [z,[x,y]], \\
[y,x]\cdot z & = (x \cdot y) \cdot z - x \cdot (y \cdot z) - (y \cdot x) \cdot z + y \cdot (x \cdot z), \\
x\cdot [y,z] & = [x \cdot y,z] + [y,x\cdot z].
\end{align*}
\end{defi}

In particular, $(V,[\, ,])$ is a Lie algebra. We obtain a second Lie bracket on $V$ by
\[
\{x,y\}:=x\cdot y-y\cdot x+[x,y].
\]
In the study of geometric structures on Lie groups the notion of a {\em post-Lie algebra structure}
was defined $2012$ in \cite{BU41}. Although it arises in a different context than a 
post-Lie algebra, it is just a reformulation of it.

\begin{defi}\label{pls}
Let $\Lg=(V, [\, ,])$ and $\Ln=(V, \{\, ,\})$ be two Lie brackets on a vector space $V$ over
$K$. A {\it post-Lie algebra structure} on the pair $(\Lg,\Ln)$ is a
$K$-bilinear product $x\cdot y$ satisfying the identities
\begin{align}
x\cdot y -y\cdot x & = [x,y]-\{x,y\}, \label{post1}\\
[x,y]\cdot z & = x\cdot (y\cdot z) -y\cdot (x\cdot z), \label{post2}\\
x\cdot \{y,z\} & = \{x\cdot y,z\}+\{y,x\cdot z\} \label{post3}
\end{align}
for all $x,y,z \in V$.
\end{defi}

Analogously we can define the notion of a {\em post-associative algebra structure} as a reformulation
of a post-associative algebra.

\begin{defi}
Let $A=(V, \pkt)$ and $B=(V, \kringel)$ be two associative products on a vector space $V$ over
$K$. A {\it post-associative algebra structure} on the pair $(A,B)$ is a pair of 
$K$-bilinear products $x\succ y$, $x\prec y$ satisfying the identities:
\begin{align}
x \pkt y - x \kringel y & = x\succ y + x\prec y, \label{postAs1}\\
(x \pkt y)\succ z & = x\succ (y\succ z), \label{postAs2}\\
x\prec (y\pkt z) & = (x\prec y)\prec z, \label{postAs3}\\
x\succ (y \kringel z) & = (x\succ y) \kringel z, \label{postAs4}\\
(x \kringel y) \prec z & = x \kringel (y\prec z), \label{postAs5}\\
(x\prec y) \kringel z & = x \kringel (y\succ z) \label{postAs6}
\end{align}
for all $x,y,z \in V$.
\end{defi}

Rewriting $x\kringel y=x\cdot y$ and $x\pkt y=x\ast y$ we see that a post-associative algebra structure
corresponds to a post-associative algebra. \\[0.2cm]
We can associate a post-Lie algebra structure to a post-associative structure as follows. Let $\Lg=(V,[\, ,])$ be the 
Lie algebra with bracket $[x,y]=x\pkt y-y\pkt x$ and $\Ln=(V,\{\, ,\})$ be the Lie algebra with
bracket $\{x,y\}=x\kringel y-y\kringel x$. Let us write a pair of associative algebras $A=(V, \pkt)$ and 
$B=(V, \kringel)$ by $(A,B)$. Then we have the following result.

\begin{lem}\label{2.5}
Let $(V, \succ, \prec )$ be a post-associative structure on a pair of associative algebras $(A,B)$. Then 
\[
x\cdot y=x\succ y-y\prec x
\]
defines a post-Lie algebra structure on the pair $(\Lg,\Ln)$.
\end{lem}

\begin{proof}
The proof is straightforward and has been given in \cite{BAG} in terms of post-Lie algebras and post-associative
algebras. Indeed, by \eqref{postAs1} we have
\begin{align*}
x\pkt y & =  x\succ y + x\prec y +  x \kringel y,\\
y\pkt x & =  y\succ x + y\prec x +  y \kringel x,
\end{align*}
and the difference yields identity \eqref{post1}. The identities \eqref{post2} and \eqref{post3} follow similarly.
\end{proof}

Note that the map 
$D_x\colon \Ln \to \Ln$ defined by 
\[
D_x(a) = x\succ a - a\prec x
\]
is a derivation of $\Ln$ for every $x \in \Ln$. \\
We can derive further identities from the above definition, in particular the following one.

\begin{lem}
Let $(V,\succ,\prec)$ be a post-associative algebra structure on a pair $(A,B)$ of associative
algebras. Then we have
\begin{align}\label{postAs7}
(x\succ y)\prec z & = x\succ (y\prec z)  
\end{align}
for all $x,y,z\in V$.
\end{lem} 

\begin{proof}
Using the identities \eqref{postAs1}-\eqref{postAs6} we have
\begin{align*}
(x\pkt y)\pkt z &  = (x\pkt y)\succ z + (x\pkt y)\prec z + (x\pkt y)\kringel z \\
                &  = (x\pkt y)\succ z + (x\succ y)\prec z + (x\prec y)\prec z + (x\kringel y)\prec z \\
                & + (x\succ y)\kringel z + (x\prec y)\kringel z + (x\kringel y)\kringel z,
\end{align*}
and similarly,
\begin{align*}
x\pkt (y\pkt z) & = x\succ (y\pkt z) + x\prec (y\pkt z) + x\kringel (y\pkt z) \\
                & = x\prec (y\pkt z) + x\succ (y\succ z) + x\succ(y\prec z) + x\succ(y\kringel z)\\ 
                & + x\kringel (y\succ z) + x\kringel(y\prec z) + x\kringel(y\kringel z). 
\end{align*}
This yields, using the associativity of $\pkt$ and $\circ$, 
\begin{align*}
0 & = (x\pkt y)\pkt z - x\pkt (y\pkt z) \\
  & = (x\pkt y)\succ z  + (x\succ y)\prec z + (x\prec y)\prec z + (x\kringel y)\prec z + (x\succ y)\kringel z \\
  & + (x\prec y)\kringel z + (x\kringel y)\kringel z - x\prec (y\pkt z) - x\succ (y\succ z) - x\succ(y\prec z) \\
  & - x\succ(y\kringel z) - x\kringel (y\succ z) - x\kringel(y\prec z) - x\kringel(y\kringel z) \\
  & = (x\succ y)\prec z - x\succ (y\succ z).
\end{align*}
\end{proof}

Note that the identity \eqref{postAs7} allows us, together with \eqref{postAs2} and \eqref{postAs3},
to view $V$ as an {\em associative $A$-bimodule}. 

\begin{defi}
Let $A$ be an algebra over a $K$-vector space $V$ and $\la\in K$. A linear operator $R\colon A\ra A$ satisfying the identity
\begin{align}
R(x)R(y) & = R\bigl(R(x)y+xR(y)+\la xy\bigr)  \label{RB}
\end{align}
for all $x,y\in A$ is called a {\em Rota--Baxter operator on $A$ of weight $\la$}, or just {\em RB-operator}.
\end{defi}

Two obvious examples are given by $R=0$ and $R=\la \id$, for an arbitrary algebra.
These are called the {\em trivial} RB-operators. Starting with an algebra $A=(V,\cdot)$ together with an RB-operator $R$ of 
nonzero weight $\lambda$ we can define a new bilinear product by
\begin{equation}\label{RBInducedAs}
x\circ y = R(x)\cdot y + x\cdot R(y) + \lambda x\cdot y.
\end{equation}
One can show that the new algebra $B=(V,\circ)$ belongs to the same variety of algebras as $A$, see \cite{BBG,GUK}.
In particular, if $A$ is an associative algebra, so is $B$. Similarly, if $A=\Lg$ is a Lie algebra, so is $B=\Ln$.  
Note that $R$ and $R + \lambda\id$ are homomorphisms from $A$ to $B$, respectively from $\mathfrak{g}$ to $\mathfrak{n}$.

We recall Proposition $2.13$ in \cite{BU59}, see also Corollary $5.6$ in \cite{BAG}.

\begin{prop}
Let $(\Ln,\{\, ,\})$ be a Lie algebra with an RB-operator of weight $1$. Then 
\[
x\cdot y = \{R(x),y\}
\]
defines a post-Lie algebra structure on the pair $(\Lg,\Ln)$, where the Lie bracket $[x,y]$ on $\Lg$ is defined
by \eqref{post1}.
\end{prop}

The corresponding result for a post-associative algebra structure has been shown in \cite{EFA}, section $4$ in terms 
of dendriform trialgebras.

\begin{prop}
Let $(B,\kringel)$ be an associative algebra with an RB-operator of weight $1$. Then 
\begin{align*}
x\succ y & = R(x)\kringel y \\
x\prec y & = x\kringel R(y)
\end{align*}
define a post-associative algebra structure on the pair $(A,B)$, where the associative product $x\pkt y$ on $A$ is 
defined by \eqref{postAs1}.
\end{prop}

Conversely we have shown in Corollary $2.15$ of \cite{BU59} that every post-Lie algebra structure 
on $(\Lg,\Ln)$, where $\Ln$ is {\em complete}, arises by an RB-operator of weight $1$ on $\Ln$.

\begin{defi}
A triple $(A,A_1,A_2)$ of algebras is called a {\em decomposition} of $A$, if  $A_1,A_2$ are subalgebras of $A$ and $A=A_1+A_2$ 
is a vector space sum of $A_1$ and $A_2$. The decomposition is called {\em proper}, if $A_1$ and $A_2$ are proper subalgebras 
of $A$. It is called {\em direct} if $A_1\cap A_2 = (0)$ and it is called  {\em semisimple} if $A,A_1,A_2$ are semisimple.
\end{defi}

We recall the following theorems by Onishchik \cite{ON62,ON69}.

\begin{thm}\label{2.11}
Let $L$ be a compact Lie algebra and $L',L''$ be two subalgebras of $L$. Let $L = S\oplus Z$, $L' = S'\oplus Z'$, and 
$L'' = S''\oplus Z''$, where $Z,Z',Z''$ are the centers and $S,S',S''$ are semisimple ideals.
Denote by $\widetilde{Z}'$ and $\widetilde{Z}''$ the projections of $Z'$ and $Z''$ on $Z$.
Then we have $L = L' + L''$ if and only if $S = S' + S''$ and $Z = \widetilde{Z}' + \widetilde{Z}''$.
\end{thm}

\begin{thm}\label{2.12}
Let $L$ be a compact Lie algebra and $L',L''$ be two subalgebras of $L$. 
All proper decompositions $(L,L',L'')$ are given as follows:
\vspace*{0.5cm}
\begin{center}
\begin{tabular}{c|c|c|c}
$L$ & $L'$ & $L''$ & $L'\cap L''$ \\
\hline
$A_{2n-1}$, $n>1$ & $C_n$ & $A_{2n-2}$, $A_{2n-2}\oplus T$ & $C_{n-1}$, $C_{n-1}\oplus T$ \\
$D_{n+1}$, $n>2$ & $B_n$ & $A_n$, $A_n\oplus T$ & $A_{n-1}$, $A_{n-1}\oplus T$ \\
$D_{2n}$, $n>1$ & $B_{2n-1}$ & $C_n$, $C_n\oplus T$, $C_n\oplus A_1$ & $C_{n-1}$, $C_{n-1}\oplus T$, $C_{n-1}\oplus A_1$\\
$B_3$ & $G_2$ & $B_2$, $B_2\oplus T$, $D_3$ & $A_1$, $A_1\oplus T$, $A_2$ \\
$D_4$ & $B_3$ & $B_2$, $B_2\oplus T$, $B_2\oplus A_1$, & $A_1$, $A_1\oplus T$, $A_1\oplus A_1$, \\
& & $D_3$, $D_3\oplus T$, $B_3$ & $A_2$, $A_2\oplus T$, $G_2$ \\
$D_8$ & $B_7$ & $B_4$ & $B_3$ \\
\end{tabular}
\end{center}
\end{thm}

In this table we have listed several decompositions in the same line. For example, the last decomposition
with $L=D_4$ is $D_4=B_3+B_3$ with intersection $L'\cap L''\cong G_2$.

\begin{defi}
A subalgebra $S$ of a Lie algebra $L$ is called a {\it reductive in} $L$, if $S$ is reductive and, moreover, 
$\mathrm{ad}(z)$ is semisimple in $\End(L)$ for every $z$ from the center of $S$.
\end{defi}

Denote a decomposition of Lie algebras $(L,L',L'')$ {\em reductive}, if $L,L',L''$ are reductive.
We have the following result by Koszul \cite{KOS}.

\begin{thm}\label{2.14}
Let $(L,L',L'')$ be a direct reductive decomposition over a field of characteristic zero,
where $L'$ and $L''$ are subalgebras reductive in $L$. Then we have  $L\cong L'\oplus L''$ and
$L',L''$ are ideals in $L$.
\end{thm}

Note that any subalgebra of a compact Lie algebra $L$ is reductive in $L$.

\section{Semisimple decompositions of associative algebras}

Similarly to the case of post-Lie algebra structures we can also ask in which cases all post-associate algebra 
structures arise from a Rota--Baxter operator. We have the following result.

\begin{thm}\label{3.1}
Let $(A,B)$ be a pair of associative algebras over an algebraically closed field $K$, where $B$ is semisimple. 
Then every post-associative algebra structure $(V,\succ,\prec)$ on $(A,B)$ arises from an RB-operator $R$ of weight $1$
with $x\succ y  = R(x)\kringel y$ and $x\prec y = x\kringel R(y)$.
\end{thm}

\begin{proof}
By the Artin-Wedderburn theorem \cite{KNA} we have $B = B_1\oplus B_2\oplus \cdots \oplus B_q$, where $B_i = M_{n_i}(K)$. 
Let us denote by $e_{ij}^a$ the matrix with entry $1$ at position $(i,j)$ and zero otherwise from the summand $M_{n_a}(K)$.
Define coefficients $l_{ija,klb}^{stc},\; r_{ija,klb}^{stc}\in K$ by
\begin{align*}
e_{ij}^a\succ e_{kl}^b & = \sum\limits_{s,t,c}r_{ija,klb}^{stc}e_{st}^c,\\
e_{ij}^a\prec e_{kl}^b & = \sum\limits_{s,t,c}l_{ija,klb}^{stc}e_{st}^c.
\end{align*}

Rewriting both sides of \eqref{postAs5} we obtain
\begin{align*}
(e_{ij}^a\kringel e_{kl}^b)\prec e_{mp}^c &  = \delta_{ab}\delta_{jk}\sum\limits_{s,t,d}l_{ila,mpc}^{std}e_{st}^d, \\
e_{ij}^a\kringel (e_{kl}^b \prec e_{mp}^c)&  = \sum\limits_{s,t,d}l_{klb,mpc}^{std}\delta_{ad}\delta_{js}e_{it}^d,  
\end{align*}
where $\delta_{ij}$ denotes the Kronecker delta. Comparing both sides we conclude that $l_{klb,mpc}^{jta} = 0$ when 
$a\neq b$ or $j\neq k$ and $l_{jla,mpc}^{jta} = l_{ula,mpc}^{uta}$ for any $1\leq j,u\leq n_a$. In the same way
by  \eqref{postAs4} we obtain 
$r_{ija,klb}^{stc} = 0$ when $b\neq c$ or $l\neq t$ and $r_{ija,klb}^{slb} = r_{ija,kvb}^{svb}$ for any $1\leq l,v\leq n_b$.
Then \eqref{postAs6} yields the equality $l_{ija,klb}^{ima} = r_{klb,mna}^{jna}$ for all $1\leq i,j,m,n\leq n_a$, $1\leq k,l\leq n_b$
and $1\leq a,b\leq q$. \\[0.2cm]
Now define a linear map $R\colon B\ra B$ by
\[
R(e_{ij}^a) = \sum\limits_{s,k,b}r_{ija,klb}^{slb}e_{sk}^b.
\]
The conditions on the coefficients $l_{ija,klb}^{stc}$ and $r_{ija,klb}^{stc}$ obtained above now ensure that we can rewrite
the post-associative algebra structure by 
\[
x\succ y = R(x)\kringel y,\; x\prec y = x\kringel R(y).
\]
The identities \eqref{postAs4}--\eqref{postAs6} are trivially satisfied, and \eqref{postAs1}--\eqref{postAs3} yield
\begin{align*}
x\pkt y & = R(x)\kringel y + x\kringel R(y) + x\kringel y, \\
R(x\pkt y)\kringel z & = R(x)\kringel R(y)\kringel z, \\
z\kringel R(x\pkt y) & = z\kringel R(x)\kringel R(y)
\end{align*}
for all $x,y,z\in V$. Since $B$ is semisimple, it has zero annihilator. Hence $R$ is an RB-operator on $B$ of 
weight $1$ and we are done.
\end{proof}

We have the following corollary.

\begin{cor}
Let $(V, \succ, \prec )$ be a post-associative structure on a pair of associative algebras $(A,B)$ over an algebraically
closed field $K$ with $A=(M_n(K),\pkt)$ and $B=(V,\kringel)$ semisimple. Then $A\cong B$ and we have either 
$x\succ y=x\prec y=0$ and $x\kringel y=x\pkt y$, or $x\pkt y = x\succ y=x\prec y=-x\kringel y$.
\end{cor}

\begin{proof}
By Theorem $\ref{3.1}$ the post-associative algebra structure on $(A,B)$ is given by an RB-operator $R$ on $B$ of weight
$1$. Both $\ker(R)$ and $\ker(R+\id)$ are ideals in $A$. Since $A$ is simple, either $\ker(R) = A$ or $\ker(R+\id) = A$.
In the first case we have $x\succ y=x\prec y=0$ and $x\kringel y=x\pkt y$, and in the second case we have
$x\pkt y = x\succ y=x\prec y=-x\kringel y$, or $\ker(R) = \ker(R+\id) = (0)$. However, the last equalities are impossible,
because then $R$ and $R+\id$ were two isomorphisms from $A$ to $B$, so that $\phi=(R+\id)R^{-1}$ were an automorphism of $B$
with $\phi(1)=1+ R^{-1}(1)\neq 1$, a contradiction.
\end{proof}

We can obtain a similar result for the case where $A$ is semisimple and $B$ is a simple algebra over $\C$.
This is based on the following decomposition theorem.

\begin{thm}\label{3.2}
There are no proper semisimple decompositions of the matrix algebra $M_n(\C)$.
\end{thm}

\begin{proof}
let $A=M_n(\C)$. Suppose that we have a proper semisimple decomposition $A = A_1 + A_2$ over $\mathbb{C}$.
By the Artin--Wedderburn theorem,
\begin{align*}
A_1 & = M_{i_1}(\mathbb{C}) \oplus \cdots \oplus M_{i_k}(\mathbb{C}),\\
A_2 & = M_{j_1}(\mathbb{C}) \oplus \cdots \oplus M_{j_l}(\mathbb{C})
\end{align*}
for some integers $i_1,\ldots ,i_k$ and $j_1,\ldots ,j_l$. This induces a decomposition of Lie algebras
$A^-=A_1^-+A_2^-$, where the Lie bracket is given by the commutator, so that
\begin{align*}
A^- & = \Lg\Ll_n(\C)\cong \Ls\Ll_n(\C)\oplus \C, \\
A_1^- & = \Ls\Ll_{i_1}(\C)\oplus \cdots \oplus  \Ls\Ll_{i_k}(\C) \oplus \C^k,\\
A_2^- & = \Ls\Ll_{j_1}(\C)\oplus \cdots \oplus  \Ls\Ll_{j_l}(\C) \oplus \C^l.
\end{align*}
Considering the compact real form for every simple Lie algebra involved we obtain
a proper decomposition $B = B_1 + B_2$ of reductive Lie algebras over $\R$, namely
\begin{align*}
B & = \Ls\Lu(n)\oplus \R,\\
B_1 & = \Ls\Lu(i_1)\oplus \cdots \oplus \Ls\Lu(i_k) \oplus \R^k,\\
B_2 & = \Ls\Lu(j_1)\oplus \cdots \oplus \Ls\Lu(j_l) \oplus \R^l.
\end{align*}

This is justified by Lemma $1.3$ from \cite{ON69}, which says that we have a~decomposition $L = L_1 + L_2$ over $\R$ 
if and only if we have a decomposition $L^{\mathbb{C}} = L_1^{\mathbb{C}} + L_2^{\mathbb{C}}$ for the complexification
of the real Lie algebras.
By Theorem $\ref{2.11}$, we obtain a semisimple decomposition
\[
\Ls\Lu(n) = \Ls\Lu(i_1)\oplus \cdots \oplus \Ls\Lu(i_k) + \Ls\Lu(j_1)\oplus \cdots \oplus \Ls\Lu(j_l).
\]
This is a contradiction to Theorem $\ref{2.12}$.  
\end{proof}

\begin{cor}
Let $(V, \succ, \prec )$ be a post-associative structure on a pair of associative algebras $(A,B)$ with
$A=(V,\pkt)$ semisimple and $B=(M_n(\C),\kringel)$. Then $A\cong B$ and we have either $x\succ y=x\prec y=0$ and
$x\kringel y=x\pkt y$, or $x\pkt y = x\succ y=x\prec y=-x\kringel y$.
\end{cor}

\begin{proof}
By Theorem $\ref{3.1}$ the post-associative algebra structure on $(A,B)$ is given by an RB-operator $R$ on $B$ of weight
$1$. Then we have a decomposition 
\[
B = \im(R) + \im(R+\id), 
\] 
since $x = R(-x) + (R+\id)(x)$ for every $x\in A$. By Theorem $\ref{3.2}$, $R$ is trivial and the claim follows.
\end{proof}

\begin{lem}\label{3.5}
Let $(A,A_1,A_2)$ be a direct semisimple decomposition of associative algebras over $\C$. Then we have $A\cong A_1\oplus A_2$.
\end{lem}

\begin{proof}
We consider the compact real forms of the direct decomposition $A^- = A_1^- + A_2^-$ of Lie algebras as in the proof of 
Theorem $\ref{3.2}$. By Theorem $\ref{2.14}$ we obtain
$A^- \cong A_1^- \oplus A_2^-$ over $\R$ and hence $A \cong A_1 \oplus A_2$ over $\C$.
\end{proof}

\section{Reductive decompositions of Lie algebras}

We can generalize Theorem $3.1$ of \cite{BU59} as follows.

\begin{thm}\label{4.1}
Suppose that there is a post-Lie algebra structure on $(\Lg,\Ln)$ over $\R$ or $\C$, where
$\Ln$ is simple and $\Lg$ is reductive. Then $\Lg$ is also simple and both $\Lg$ and $\Ln$ are isomorphic.
\end{thm}

\begin{proof}
Since $\Ln$ is complete there is a bijection between post-Lie algebra structures 
on $(\Lg,\Ln)$ and RB-operators on $\Ln$ of weight $1$. So let $R$ be an RB-operator of weight $1$ on $\Ln$ 
corresponding to the given post-Lie algebra structure on $(\Lg,\Ln)$. We have a proper reductive decomposition 
$\Ln= \im(R) + \im(R+\id)$. Furthermore $\Lg = \La\oplus \Lb\oplus \Lc$ is the sum of three ideals with $\La = \ker(R)$ and
$\Lb = \ker(R+\id)$. Assume that the field is $\C$. Consider the decomposition  $\Ln= \im(R) + \im(R+\id)$ over the 
real numbers, splitting it into semisimple and abelian parts, see Theorem $\ref{2.11}$. We obtain a proper semisimple
decomposition
\[
\Ln_{\R}=\Ls_1+\Ls_2,
\]
where $\Ln_{\R}$ is the compact real form of $\Ln$ and $\Ls_1,\Ls_2$ are the semisimple parts of $\im(R)$ and $\im(R+\id)$ 
considered over $\R$ respectively. Suppose that $\Lc$ is abelian. Then the decomposition
$\Ln_{\R}=\Ls_1\oplus \Ls_2$ is direct, which is impossible by Theorem $\ref{2.12}$. Hence $\Lc$ is non-abelian and non-zero,
so that  $\im(R)$ and $\im(R+\id)$ contain a pair of isomorphic simple summands. Suppose that $\La$ and $\Lb$ are non-abelian.
Then $\im(R)$, $\im(R+\id)$ and hence $\Ls_1,\Ls_2$ have at least two simple summands. This is again impossible by
Theorem $\ref{2.12}$. On the other hand, suppose that $\La$ and $\Lb$ are abelian. Then we have $\Ls_1\cong \Ls_2$. By
Theorem $\ref{2.12}$ this is only possible in the case $D_4 = B_3 + B_3$ over $\R$. With $\La = \C^k$, 
$\Lb = \C^l$ and $\Lc = B_3\oplus \C^m$ we may rewrite the decomposition $\Ln= \im(R) + \im(R+\id)$ as 
\[
D_4 = (B_3\oplus\C^{l+m}) + (B_3\oplus\C^{k+m}).
\]
We have $28=\dim(\C^k\oplus \C^l\oplus B_3\oplus \C^m)$, so that $k+l+m = \dim(D_4) - \dim(B_3) = 7$. Hence 
at least one of the summands from the above decomposition contains a subalgebra $B_3\oplus \C^4$. So the centralizer
of $B_3$ in $D_4$ is at least $4$-dimensional. However, according to \cite[Table~11]{MIN} it can be at most $1$-dimensional.
This gives a contradiction. \\
Finally, suppose that only one of $\La$ and $\Lb$ is abelian. 
We may suppose that $\Ls_1$ is simple. But then $\Ls_1$ is a proper ideal in $\Ls_2$, contradicting again Theorem $\ref{2.12}$.
Over $\R$, we can complexify the RB-operator and the decomposition obtaining a proper reductive decomposition
of a simple Lie algebra over $\C$.
\end{proof}

\begin{rem}
The above theorem has an easier proof for the exceptional Lie algebras $G_2$ and $E_8$, not using the classification
by Onishchik. For $G_2$, the statement follows from the description of all its subalgebras \cite{DRE,MAY}. 
For $E_8$, it follows from the information on the centralizers of all its semisimple subalgebras \cite[Table 14]{MIN}.
\end{rem}

\begin{cor}
Let $L$ be a real or complex simple Lie algebra of exceptional type, or of type $B_n, C_n$ with $n\ge 3$.
Then there is no proper reductive decomposition $(L,L',L'')$.
\end{cor}

\begin{proof}
This follows as above from Theorem $\ref{2.12}$.
\end{proof}

Recall that $Z^1(\Lg,M)$ denotes the space of $1$-cocyles for the Lie algebra cohomology of $\Lg$ with 
a $\Lg$-module $M$, and $B^1(\Lg,M)$ the space of $1$-coboundaries.

\begin{lem}\label{4.4}
Let $(V,\cdot)$ be a post-Lie algebra structure on a pair $(\Lg,\Ln)$ arising from an RB-operator $R$ of
weight $1$ on $\Ln$. Let $M$ be a $\Lg$-module. Then 
\[
x\cdot_{\mathfrak{g}}m = R(x)\cdot_{\mathfrak{n}}m
\]
for $x\in V$, $m\in M$ defines a $\Lg$-module structure on $M$. For a $1$-cocycle $d\in Z^1(\Ln,M)$
the linear map $d_R$ defined by $d_R(x)=d(R(x))$ is a $1$-cocycle in $Z^1(\Lg,M)$.
\end{lem}

\begin{proof}
By \eqref{RB} we have
\begin{align*}
[x,y]\cdot_{\mathfrak{g}}m &  = R([x,y])\cdot_{\mathfrak{n}}m \\
  & = \{R(x),R(y)\}\cdot_{\mathfrak{n}}m \\
  & = R(x)\cdot_{\mathfrak{n}}(R(y)\cdot_{\mathfrak{n}}m) - R(y)\cdot_{\mathfrak{n}}(R(x)\cdot_{\mathfrak{n}}m) \\
  & = x\cdot_{\mathfrak{g}}(y\cdot_{\mathfrak{g}} m) - y\cdot_{\mathfrak{g}}(x\cdot_{\mathfrak{g}} m).
\end{align*}
Hence $M$ is a $\Lg$-module. The map $d_R$ is a $1$-cocycle since we have
\begin{align*}
d_R([x,y]) & = d(R([x,y]) \\
 &  = d(\{R(x),R(y)\}) \\
 & = d(R(x))\cdot_{\mathfrak{n}}R(y) + R(x)\cdot_{\mathfrak{n}}d(R(x)) \\
 & = d_R(x)\cdot_{\mathfrak{g}}y + x\cdot_{\mathfrak{g}}d_R(y).
\end{align*}
\end{proof}

In the same way one can also prove the following lemma.

\begin{lem}\label{4.5}
Let $(V,\succ, \prec)$ be a post-associative algebra structure on a pair $(A,B)$ arising from an RB-operator $R$ of
weight $1$ on $B$. Let $M$ be a $B$-bimodule. Then 
\begin{align*}
x\cdot_{A}m & = R(x)\cdot_{B}m, \\
m\cdot_{A}x & = m\cdot_{B}R(x)
\end{align*}
for $x\in V$, $m\in M$ defines an $A$-bimodule structure on $M$. For a $1$-cocycle $d\in Z^1(B,M)$ the linear
map $d_R$ defined by  $d_R(x)=d(R(x))$ is a $1$-cocycle in $Z^1(A,M)$.
\end{lem}

\begin{thm}\label{4.6}
Let $(V,\cdot)$ be a post-Lie algebra structure on a pair $(\Lg,\Ln)$ over a field of characteristic zero, arising
from an RB-operator of weight $1$ on $\Ln$. Suppose that $\Lg$ is semisimple. Then $\Ln$ is semisimple, too.
\end{thm}
\begin{proof}
Consider a $\Ln$-module $M$ and a $1$-cocycle $d\in Z^1(\Ln,M)$. By Lemma $\ref{4.4}$ it follows that $d_R\in Z^1(\Lg,M)$,
where $d_R(x)=d(R(x))$. Since $\Lg$ is semisimple, $d_R\in B^1(\Lg,M)$ by the first Whitehead lemma. Hence there exist 
$b,c\in M$ such that
\begin{align*}
d(R(x)) & = bx, \\
d(-(R+\id)(x)) & = cx.
\end{align*}
Hence $d(x) = -(b+c)x$ and $d\in B^1(\Ln,M)$. It follows that $\Ln$ is semisimple \cite{HOC}.
\end{proof}

In the same way one can also prove the analogous statement for post-associative algebra structures.

\begin{thm}
Let $(V,\succ,\prec)$ be a post-associative algebra structure on a pair $(A,B)$ 
over a field of characteristic zero,  arising from an RB-operator of weight $1$ on $B$. 
Suppose that $A$ is semisimple. Then $B$ is semisimple, too.
\end{thm}

Since every post-Lie algebra structure on $(\Lg,\Ln)$, where $\Ln$ is complete arises by an RB-operator of weight $1$
on $\Ln$, Theorem $4.6$ immediately implies the following corollary, which is Proposition $3.8$ in \cite{BU59}.

\begin{cor}\label{4.8}
Let $(V,\cdot)$ be a post-Lie algebra structure on a pair $(\Lg,\Ln)$ over a field of characteristic zero, where
$\Lg$ is semisimple and $\Ln$ is complete. Then $\Ln$ is semisimple.
\end{cor}

\begin{rem}\label{4.9}
The proof of Proposition $3.8$ in \cite{BU59} unfortunately is not correct. It used Proposition $3.6$ claiming
that a Lie algebra $L$, which is the sum of two complex semisimple subalgebras, is semisimple. However, this is
not true as the next example shows. Above we have given a new proof, which also generalizes the result to arbitrary
fields of characteristic zero. Below we will also give a new proof of Proposition $3.7$ in \cite{BU59}, which relied on
Proposition $3.6$, too. 
\end{rem}

\begin{ex}
Let $L=\Ls\Ll_n(\C)\rtimes V(n)$, where $n\ge 2$ and $V(n)$ is the irreducible representation of $\Ls\Ll_n(\C)$, considered
as abelian Lie algebra. Then $L$ is a perfect, non-semisimple Lie algebra, which is the vector space sum of two simple
Lie algebras as follows. Let $(e_1,\cdots, e_n)$ be a basis of $V(n)$ and let $x = e_1+\ldots+e_n\in V(n)$. 
Then $\ad (x)$ is nilpotent with $\ad(x)^2=0$. Consider the automorphism $\varphi = \exp(\ad(x)) = \id + \,\ad(x)$ of $L$. 
We obtain the decomposition
\[
L=\Ls\Ll_n(\C)+\phi(\Ls\Ll_n(\C)).
\]
\end{ex}

Also in the associative case the sum of two semisimple algebras need not be semisimple. For a classification
of associative algebras being a sum of two simple subalgebras see \cite{BAH}. \\[0.5cm]
Given a post-Lie algebra structure $(V,\cdot)$ on a pair $(\Lg,\Ln)$, which is defined by an RB-operator $R$ of weight $1$
on $\Ln$, we can define a sequence of Lie brackets on $V$ by
\begin{align*}
[x,y]_0 & = \{x,y\},\\
[x,y]_{i+1} & = [R(x),y]_i + [x,R(y)]_i+[x,y]_i, 
\end{align*}
for all $i\geq 0$, see \cite{BU59}. Then $R$ defines a post-Lie algebra structure on each pair $(\Lg_{i+1},\Lg_i)$.
We have $[x,y]_1=[x,y]$, and both $R$ and $R+\id$ are Lie algebra homomorphisms from $\Lg_{i+1}$ to $\Lg_i$.
Hence we obtain a composition of homomorphisms
\[
\mathfrak{g}_i \xrightarrow[R+\id]{R} \mathfrak{g}_{i-1}
 \xrightarrow[R+\id]{R} \cdots  \xrightarrow[R+\id]{R} \mathfrak{g}_{0}.
\]
So $\ker(R^i)$ and $\ker((R+\id)^i)$ are ideals in $\mathfrak{g}_j$ for all $1\leq i\leq j$. \\[0.2cm]
We may define the same sequence of algebras in the associative case. Given a post-associative algebra structure 
$(V,\succ,\prec)$ on $(A,B)$ defined by an RB-operator $R$ of weight $1$ on $B$,
denote by $A_i$ the associative algebra on $V$ defined by
\begin{align*}
x\pkt_0 y & = x\kringel y,\\
x\pkt_{i+1} y & = R(x)\pkt_i y + x\pkt_i R(y) + x\pkt_i y, 
\end{align*}
for all $i\geq 0$. Then $R$ defines a post-associative algebra structure on each pair $(A_{i+1},A_i)$.
We have $x\pkt_1 y = x\pkt y$ and we get the similar composition of homomorphisms
\[
A_i \xrightarrow[R+\id]{R} A_{i-1}
 \xrightarrow[R+\id]{R} \cdots  \xrightarrow[R+\id]{R} A_0.
\]
So $\ker(R^i)$ and $\ker((R+\id)^i)$ are ideals in $A_j$ for all $1\leq i\leq j$. \\[0.2cm]
We can now give a new proof of Proposition $3.7$ in \cite{BU59}.

\begin{cor}\label{4.11}
Let $(V,\cdot)$ be a post-Lie algebra structure on a pair $(\Lg,\Ln)$ over a field of characteristic zero,
defined by an RB-operator of weight $1$ on $\Ln$. Assume that $\mathfrak{g}_n$ is semisimple, where $n = \dim(V)$. 
Then all $\mathfrak{g}_i$ are isomorphic to $\Ln$.
\end{cor}

\begin{proof}
Since $\ker(R^n)$ and $\ker((R+\id))^n$ are ideals in $\mathfrak{g}_n$ it follows from Proposition $3.4$ in \cite{BU59} that 
$\mathfrak{g}_n = \ker(R^n)\oplus \ker((R+\id))^n$. Then Corollary $3.5$ in \cite{BU59} implies that 
\[
\mathfrak{g}_i = \ker(R^n)\dotplus \ker((R+\id))^n,
\]
where the algebras $\ker(R^n)$ in $\Lg_i$ are isomorphic for all $i\ge 0$. The same holds for $\ker((R+\id))^n$.
So, $\ker(R^n)$ and $\ker((R+\id))^n$ are semisimple subalgebras in $\Lg_i$ for all $i$.
Since $\mathfrak{g}_n$ is semisimple, $\mathfrak{g}_{n-1}$ is semisimple by Theorem $\ref{4.6}$. Iterating this we see
that $\mathfrak{g}_i$ are semisimple for all $i\ge 0$. By Theorem $\ref{2.14}$ we obtain
$\Lg_i \cong \ker(R^n)\oplus \ker((R+\id))^n \cong \mathfrak{g}_n$. Since $\Lg_0\cong \Ln$ we are done.
\end{proof}

Similarly we can prove the corresponding result for post-associative algebra structures. Then we should use
Lemma $\ref{3.5}$ over $\C$ instead of Theorem  $\ref{2.14}$. 

\begin{cor}
Let $(V,\succ, \prec)$ be a post-associative algebra structure on a pair $(A,B)$ over $\C$, defined by an RB-operator 
of weight $1$ on $B$. Assume that $A_n$ is semisimple, where $n = \dim(V)$. Then all $A_i$ are isomorphic to $B$.
\end{cor}

Given a post-associative algebra structure on a pair $(A,B)$ of semisimple associative algebras, in general we do not know
whether $A$ and $B$ have to be isomorphic or not. The same question is open for post-Lie algebra structures.
In some cases we have a positive answer. Here is another such case.

\begin{prop}
Let $(V,\succ,\prec)$ be a post-associative algebra structure on a pair of complex semisimple algebras $(A,B)$. 
Suppose that either $A$ or $B$ is commutative. Then $A$ and $B$ are isomorphic. 
\end{prop}

\begin{proof}
By Theorem $\ref{3.1}$ the post-associative algebra structure on $(A,B)$ is defined by an RB-operator $R$ of weight
$1$ on $B$. First suppose that $A$ is commutative. Then it follows that $A\cong B$ by \cite{GUB}.
Secondly, let $B$ be commutative.
So it is isomorphic to a direct sum of copies of $\C$. Since  $R$ is an RB-operator on $B = \C^n$ with $n = \dim(V)$, 
we have ${\rm Spec}(R)\subset\{0,-1\}$, see \cite{GUB}. So, $B = \ker(R)^n \dotplus \ker(R+\id)^n$.
By linearization and dimension reasons, $\ker(R)^n\cong \mathbb{C}^m$ and
$\ker(R+\id)^n\cong \mathbb{C}^{n-m}$ for some $m$ with $0\le m\le n$. By Lemma $\ref{3.5}$ we obtain 
\[
A = \ker(R)^n \dotplus \ker(R+\id)^n\cong \mathbb{C}^m\oplus \mathbb{C}^{n-m}\cong B.
\]
\end{proof}

Finally we obtain a result on  post-associative algebra structures by using
a classical result of Kegel \cite{KEG} on nilpotent decompositions. 

\begin{prop}
Let $(V,\succ,\prec)$ be a post-associative algebra structure on a pair of algebras $(A,B)$ over an 
algebraically closed field $F$, where $B$ is semisimple. Then $A_i$ is not nilpotent for all $i\ge 0$.
\end{prop}

\begin{proof} 
By Theorem $\ref{3.1}$ the post-associative algebra structure on $(A,B)$ arises from
an RB-operator $R$ of weight $1$ on $B$. Assume that $A_i$ is nilpotent for some $i\ge 0$. 
Then 
\[
A_{i-1} = R(A_i) + (R+\id)(A_i) 
\]
is also nilpotent as a sum of two nilpotent associative subalgebras \cite{KEG}. Iterating this we obtain that 
$A_0=B$ is nilpotent, a contradiction.
\end{proof}

\section*{Acknowledgments}
Dietrich Burde is supported by the Austrian Science Foun\-da\-tion FWF, grant P28079
and grant I3248. Vsevolod Gubarev acknowledges support by the Austrian Science Foun\-da\-tion FWF,
grant P28079. We are grateful to A. Petukhov for helpful discussions.

\end{document}